\documentclass[12pt]{article}
\usepackage{amssymb, amscd, amsmath, amsfonts, amsthm}
\usepackage[dvipsnames]{xcolor}
\usepackage{verbatim}
\usepackage{authblk}
\usepackage{fullpage}

\newdimen\squaresize \squaresize=12pt
\newdimen\thickness \thickness=0.4pt
 
\def\square#1{\hbox{\vrule width \thickness
   \vbox to \squaresize{\hrule height \thickness\vss
      \hbox to \squaresize{\hss#1\hss}
    \vss\hrule height\thickness}
\unskip\vrule width \thickness}
\kern-\thickness}
 
\def\vsquare#1{\vbox{\square{$#1$}}\kern-\thickness}

\def\thisbox#1{\kern-.09ex\fbox{#1}}
\def\downbox#1{\lower1.200em\hbox{#1}}

\def\boldentry#1{\textcolor{Red}{\textbf{#1}}}

\usepackage{tikz}
\usetikzlibrary{calc}
\usepackage{ifthen}
\newcommand{\tikztableau}[2][scale=0.6,every node/.style={font=\small}]{
    \def\newtableau{#2}
    \begin{array}{c}
    \begin{tikzpicture}[#1]
    \coordinate (x) at (-0.5,0.5);
    \coordinate (y) at (-0.5,0.5);
    \foreach \row in \newtableau {
        \coordinate (x) at ($(x)-(0,1)$);
        \coordinate (y) at (x);
        \foreach \entry in \row {
            \ifthenelse{\equal{\entry}{X}}
               {
                \node (y) at ($(y) + (1,0)$) {};
                \fill[color=gray!10] ($(y)-(0.5,0.5)$) rectangle +(1,1);
                \draw[color=gray] ($(y)-(0.5,0.5)$) rectangle +(1,1);
               }
               {
                \ifthenelse{\equal{\entry}{\boldentry X}}
                   {
                    \node (y) at ($(y) + (1,0)$) {};
                    \fill[color=gray] ($(y)-(0.5,0.5)$) rectangle +(1,1);
                    \draw ($(y)-(0.5,0.5)$) rectangle +(1,1);
                   }
                   {
                    \node (y) at ($(y) + (1,0)$) {\entry};
                    \draw ($(y)-(0.5,0.5)$) rectangle +(1,1);
                   }
               }
            }
        }
    \end{tikzpicture}
    \end{array}}

\newcommand{\tikztableausmall}[1]{\tikztableau[scale=0.45,every node/.style={font=\rm\small}]{#1}}

\usepackage[latin1]{inputenc}
\usepackage{subfigure}

\begin{document}

\author{Chris Berg, Franco Saliola and Luis Serrano} 

\affil{Laboratoire de combinatoire et d'informatique math\'ematique\\ Universit\'e du Qu\'ebec \`a Montr\'eal} 

\title{The down operator and expansions of near rectangular $k$-Schur functions}

\maketitle

\newtheorem{thm}{Theorem}[section]
\newtheorem{remark}[thm]{Remark}
\newtheorem{defn}[thm]{Definition}
\newtheorem{lemma}[thm]{Lemma}
\newtheorem{obs}[thm]{Observation}
\newtheorem{cor}[thm]{Corollary}
\newtheorem{eg}[thm]{Example}
\newtheorem{prop}[thm]{Proposition}
\newcommand{\K}{\{0,1,\dots,k\}}
\newcommand{\Pk}{\mathcal{P}^{(k)}}
\newcommand{\Ckp}{\mathcal{C}^{(k+1)}}
\newcommand{\Lk}{\Lambda^{(k)}}
\newcommand{\slk}{s_\lambda^{(k)}}
\newcommand{\tr}{z}
\def\mfc{{\mathfrak c}}
\def\tnu{{\tilde \nu}}
\def\tv{{\tilde v}}
\newcommand\kschur[1][k]{\mathfrak{s}^{(#1)}}\def\bfa{{\bf a}}
\def\bfh{{\bf h}}
\def\uu{{\bf u}}
\def\RRR{{\mathfrak{R}}}
\def\win{{\bf win}}
\def\transpose{{t}}
\def\READ#1{\operatorname{\textrm{\bf word}}(#1)}
\def\centroid{{G}}
\def\lbl{{L}}
 
\newcommand{\Z}{\mathbb{Z}}
\newcommand{\R}{\mathbb{R}}
\newcommand{\C}{\mathbb{C}}
 
\newcommand{\Waf}{W_{\textrm{af}}}
\newcommand{\Wfin}{W_{\textrm{fin}}}
\newcommand{\Iaf}{I_{\textrm{af}}}
\newcommand{\Aaf}{A_{\textrm{af}}}
\newcommand{\Qaf}{Q_{\textrm{af}}}
\newcommand{\Paf}{P_{\textrm{af}}}
\newcommand{\Ifin}{I_{\textrm{fin}}}
\newcommand{\Afin}{A_{\textrm{fin}}}
\newcommand{\Qfin}{Q_{\textrm{fin}}}
\newcommand{\Pfin}{P_{\textrm{fin}}}
\newcommand{\asscovercoroot}{\alpha_{vw}^\vee}
\newcommand{\Wext}{W_{\textrm{ext}}}
\newcommand{\hstar}{\mathfrak{h}^*}
\newcommand{\h}{\mathfrak{h}}
\newcommand{\haf}{\mathfrak{h}_{\textrm{af}}}
\newcommand{\hstaraf}{\mathfrak{h}_{\textrm{af}}^*}
\newcommand{\hstarfin}{\mathfrak{h}_{\textrm{fin}}^*}
\newcommand{\hfin}{\mathfrak{h}_{\textrm{fin}}}
\newcommand{\Hom}{\textrm{Hom}}
\newcommand{\realrootsaf}{\Phi_{\textrm{re}}}
\newcommand{\realrootsfin}{\Phi_{\textrm{fin}}}
\newcommand{\level}{\textrm{lev}}
\newcommand{\Gr}{\textrm{Gr}_G}
 
\newcommand{\levoneaction}{\tiny{\textcircled{\raisebox{-.4pt} 1}}}
\newcommand{\levzeroaction}{\tiny{\textcircled{\raisebox{-.4pt} 0}}}
\newcommand{\weakcover}{\mathrel{\prec\!\!\!\cdot}}
\newcommand{\weakorder}{\prec}

\begin{abstract}
We prove that the Lam-Shimozono ``down operator'' on the affine Weyl group
induces a derivation of the affine Fomin-Stanley subalgebra. We use this to verify a conjecture of Berg, Bergeron, Pon
and Zabrocki describing the expansion of non-commutative $k$-Schur functions of ``near
rectangles'' in the affine nilCoxeter algebra. Consequently, we obtain a
combinatorial interpretation of the corresponding $k$-Littlewood--Richardson
coefficients.
\end{abstract}

\section{Introduction}

$k$-Schur functions were first introduced by Lapointe, Lascoux and Morse \cite{LLM} in the study of Macdonald polynomials. Since then, their study has flourished (see for instance \cite{Lam1, LS, LLMS, LM0, LM1, LM2}). In particular they have been realized as Schubert classes for the homology of the affine Grassmannian.
This was done by identifying the algebra of $k$-Schur functions with the affine
Fomin-Stanley subalgebra of the affine nilCoxeter algebra $\mathbb{A}$
\cite{Lam2}. Under this identification, the image of a $k$-Schur function is called a non-commutative $k$-Schur function. A natural question is to ask for the expansion of a non-commutative $k$-Schur
function in terms of the standard basis of $\mathbb{A}$, which is indexed by
affine permutations. 

A important related problem is to describe the multiplicative structure constants of
the $k$-Schur functions, called the \emph{$k$-Littlewood--Richardson}
coefficients due to the similarity with the classical problem of multiplying
Schur functions. It was pointed out in \cite{Lam0} that the
$k$-Littlewood--Richardson coefficients are the same coefficients that appear
in the expansion of a non-commutative $k$-Schur function in the standard basis of $\mathbb{A}$
(see Section \ref{ss:expansions}). Hence, results that give such expansions
also give information about the $k$-Littlewood--Richardson coefficients. This
paper is one such example; others are \cite{Lam2, BSZ, BBTZ, BBPZ}.

In early 2011, Berg, Bergeron, Pon and Zabrocki conjectured an expansion for
non-commutative $k$-Schur functions indexed by a $k$-rectangle $R$ minus its unique removable
cell. Their conjecture combined ideas coming from two groups: Pon's \cite{P}
description of the generators of the affine Fomin-Stanley subalgebra for
arbitrary affine type; and Berg, Bergeron, Thomas and Zabrocki's \cite{BBTZ}
expansion of $\kschur_R$.

This paper initiates the study of certain operators on the affine nilCoxeter algebra
which stabilize the affine Fomin-Stanley subalgebra. We study one particular
family of operators introduced by Lam and Shimozono \cite{LS} and prove that
they are derivations of the affine Fomin-Stanley subalgebra (Theorem
\ref{thm:derivation}). As an application, we use these operators to prove the conjecture of Berg,
Bergeron, Pon and Zabrocki and provide a combinatorial interpretation for the
corresponding $k$-Littlewood--Richardson coefficients (Theorem
\ref{thm:mainthm}). Further properties of such operators and their applications
to  $k$-Schur functions will be developed in a companion article.

\section{$k$-Combinatorics}
 
 In this section, we recall the required terminology associated to the affine type $A$ root system, the affine Weyl group, the connection with bounded partitions and core partitions and the definition of non-commutative $k$-Schur functions.
We work with the affine type $A$ root system $A_k^{(1)}$. Much of this introduction is borrowed from \cite{BBPZ} which in turn was borrowed from \cite{Shim}.
 
\subsection{Affine symmetric group}
$I = \{0,1,\dots, k \}$ will denote the set of nodes of the corresponding Dynkin diagram. We say two nodes $i,j \in I$ are adjacent if $i-j = \pm 1 \mod (k+1)$.
 
We let $W$ denote the \emph{affine symmetric group} with generators $s_i$ for $i \in I$, and relations  $s_i^2 = 1$, $s_is_j = s_j s_i$, when $i$ and $j$ are not adjacent, and $s_i s_j s_i = s_j s_i s_j$ when $i$ and $j$ are adjacent.  An element of the affine symmetric group may be expressed as a word in the generators $s_i$. Given the relations above, an element of the affine symmetric group may have multiple \emph{reduced words}, words of minimal length which express that element. The \emph{length} of $w$, denoted $\ell(w)$, is the number of generators in any reduced word of $w$.
 
 The \emph{Bruhat order} on affine symmetric group elements is a partial order where $v < w$ if there is a reduced word for $v$ that is a subword of a reduced word for $w$.  If $v < w$ and $\ell(v) = \ell(w) - 1$, we write $v \lessdot w$.
There is another order on $W$, called the \textit{left weak order}, which is defined by the covering relation $v \weakcover w$ if $w = s_i v $ for some $i$ and $\ell(v) = \ell(w) -1$.

 For $j \in I$, we denote by $W_j$ the subgroup of $W$ generated by the elements $s_i$ with $i \neq j$.  We denote by $W^j$ the set of minimal length representatives of the cosets $W/W_j$.

\subsection{Roots and weights} 
Associated to the affine Dynkin diagram of type $A_k^{(1)}$ we have a root datum, which consists of a free $\Z$-module $\h$, its dual lattice $\hstar = \Hom(\h, \mathbb{Z})$, a pairing $\langle \cdot , \cdot \rangle: \h \times \hstar \to \Z$ given by $\langle \mu,\lambda \rangle = \lambda(\mu)$, and sets of linearly independent elements $\{ \alpha_i \mid i \in I\} \subset \hstar$ and $\{\alpha_i^\vee \mid i \in I\} \subset \h$ such that
\begin{equation}
\langle \alpha_i^\vee, \alpha_j \rangle = \left\{
        \begin{array}{ll}
                2  & \mbox{if } i = j; \\
                -1 & \mbox{if } i \textrm{ and } j \textrm{ are adjacent;} \\
                0 & \textrm{else.}
        \end{array}
\right.
\end{equation}
The $\alpha_i$ are known as \emph{simple roots}, and the $\alpha_i^\vee$ are \emph{simple coroots}.  The spaces $\h_\R = \h \otimes \R$ and $\hstar_\R = \hstar \otimes \R$ are the \emph{coroot} and \emph{root spaces}, respectively.

Given a simple root $\alpha_i$, we have actions of $W$ on $\h_\R$ and $\hstar_\R$
defined by the action of the generators of $W$ as
\begin{align}
s_i(\lambda) &= \lambda - \langle \alpha_i^\vee, \lambda \rangle \alpha_i \quad \textrm{for } i \in I, \lambda \in \hstar_\R;\label{eqn:action}\\
s_i(\mu) &= \mu - \langle  \mu, \alpha_i \rangle \alpha_i^\vee \quad \textrm{for } i \in I, \mu \in \h_\R.
\end{align}
The action of $W$ satisfies 
\begin{equation}\label{eqn:w} \langle w(\mu), w(\lambda) \rangle = \langle \mu,\lambda \rangle \end{equation}
for all $\mu \in \h_\R$, $\lambda \in \hstar_\R$ and $w \in W$.
 
The set of \emph{real roots} is $\realrootsaf = W \cdot \{\alpha_i \mid i \in I\}$.  Given a real root $\alpha = w(\alpha_i)$, we have an  \emph{associated coroot} $\alpha^\vee = w(\alpha_i^\vee)$ and an  \emph{associated reflection} $s_\alpha = w s_i w^{-1}$ (these are well-defined, and independent of choice of $w$ and $i$).
For a Bruhat covering $v \lessdot w$, there exists a unique root $\alpha_{v,w}$ satisfying the equation $v^{-1} w = s_{\alpha_{v,w}}$. We denote by $\alpha_{v,w}^\vee$ the coroot corresponding to the root $\alpha_{v,w}$.

We let \[\delta =  \alpha_0 + \cdots +  \alpha_n \in \hstar\] denote the \emph{null root}.  
 On the dual side, we let \[c = \alpha_0^\vee + \cdots + \alpha_n^\vee \in \h\] denote the canonical central element. Under the action of $W$, we have $w(\delta) = \delta$ and $w(c) = c$ for $w \in W$.

The action by $W$ preserves the \emph{root lattice} $Q = \bigoplus_{i \in I} \Z \alpha_i$ and \emph{coroot lattice} $Q^\vee = \bigoplus_{i \in I} \Z \alpha_i^\vee$. 

 The \emph{fundamental weights} are the elements 
$\Lambda_i \in \hstar_\R$ satisfying
$\langle \alpha_j^\vee, \Lambda_i \rangle = \delta_{ij}$ for $i,j \in I$ for $i,j \in I$. They generate the \emph{weight lattice} $P = \bigoplus_{i \in I} \Z \Lambda_i$. We let $P^+ = \bigoplus_{i\in I} \mathbb{Z}_{\geq 0} \Lambda_i$ denote the \emph{dominant weights}.  
The \emph{fundamental coweights} are $\{\Lambda_i^\vee \in \h_\R \mid \alpha_i(\Lambda_j^\vee) = \delta_{ij}\,   \hbox{ for } i,j \in I \}$.
  They generate the \emph{coweight lattice} $P^\vee = \bigoplus_{i \in I} \Z \Lambda_i^\vee$.
  
 
\subsection{$k$-bounded partitions, $(k+1)$-cores and affine Grassmannian elements}

Let $\lambda$ be a partition. To each box $(i,j)$ (row $i$, column $j$) of the
Young diagram of $\lambda$, we associate its \textit{residue} defined by
$c_{(i,j)} = (j-i) \mod{(k+1)}$. We let $\Pk$ denote the set of
\emph{$k$-bounded partitions}, namely the partitions $\lambda = (\lambda_1,
\lambda_2, \dots)$ whose first part $\lambda_1$ is at most $k$.

A \emph{$p$-core} is a partition that has no removable rim hooks of length $p$. Lapointe and Morse \cite[Theorem 7]{LM1} showed a bijection between the set $\Pk$ and the set of $(k+1)$-cores. Following their notation, we let $\mathfrak{c}(\lambda)$ denote the $(k+1)$-core corresponding to the partition $\lambda$, and $\mathfrak{p}(\mu)$ denote the $k$-bounded partition corresponding to the $(k+1)$-core $\mu$.
We will also use $\Ckp$ to represent the set of all $(k+1)$-cores.
 
 $W$ acts on $\Ckp$. Specifically, if $\lambda$ is a $(k+1)$-core then 
 \[s_i \lambda =  
\left\{ \begin{array}{ll} \lambda \cup \{\textrm{addable residue } i \textrm{ cells}\} & \textrm{ if } \lambda \textrm{ has an addable cell of residue } i, \\
\lambda \setminus \{ \textrm{removable residue } i \textrm{ cells}\} & \textrm{ if } \lambda \textrm{ has  a removable cell of residue } i, \\ \lambda & \textrm{ otherwise. } \end{array} \right.\]

The \emph{affine Grassmannian elements} are the elements of $W^0$. These
are naturally identified with $(k+1)$-cores in the following way: to a core $\lambda \in \Ckp$, we associate the unique element $w \in W^0$ for which $w\emptyset = \lambda$. For a $k$-bounded partition $\mu$, we let  $w_\mu$ denote the  element of $W^0$ which satisfies  $w_\mu \emptyset = \mathfrak{c}(\mu)$. More details on this can be found in \cite{BB}.

\begin{eg}
The diagram of the $4$-core $\lambda=(5,2,1)$ augmented with its residues,
together with the diagrams of the $4$-cores $s_1 \lambda$ and $s_0 \lambda$:
\begin{center}
    $\lambda = \tikztableausmall{{0,1,2,3,0},{3,0},{2}} \quad s_1 \lambda = \tikztableausmall{{0,1,2,3,0,1},{3,0,1},{2},{1}} \quad s_0 \lambda = \tikztableausmall{{0,1,2,3},{3},{2}}.$
\end{center}
\end{eg}

\subsection{The non-commutative $k$-Schur functions}

Let $\mathbb{F} = \mathbb{C}((t))$ and $\mathbb{O} = \mathbb{C}[[t]]$.  The \emph{affine Grassmannian} may be given by $\Gr := G(\mathbb{F})/G(\mathbb{O})$.  $\Gr$ can be decomposed into \emph{Schubert cells} $\Omega_w = \mathcal{B}wG(\mathbb{O}) \subset G(\mathbb{F})/G(\mathbb{O})$, where $\mathcal{B}$ denotes the Iwahori subgroup and $w \in W^0$, the set of \textit{Grassmannian} elements of $W$.  The Schubert varieties, denoted $X_w$, are the closures of $\Omega_w$, and we have $\Gr = \sqcup \Omega_w = \cup X_w$, for $w \in W^0$.  The homology $H_*(\Gr)$ and cohomology $H^*(\Gr)$ of the affine Grassmannian have corresponding Schubert bases, $\{ \xi_w\}$ and $\{\xi^w\}$, respectively, also indexed by Grassmannian elements.  It is well-known that $\Gr$ is homotopy-equivalent to the space $\Omega K$ of based loops in $K$ (due to Quillen, see \cite[\textsection 8]{ps86} or \cite{mit88}).  The group structure of $\Omega K$ gives $H_*(\Gr)$ and $H^*(\Gr)$ the structure of dual Hopf algebras over $\mathbb{Z}$.

The \emph{nilCoxeter algebra} $\mathbb{A}$ may be defined via generators and relations  with generators $\uu_i$ for $i \in I$, and relations $\uu_i^2 = 0$, $\uu_i\uu_j = \uu_j \uu_i$ when $i$ and $j$ are not adjacent and $\uu_i \uu_j \uu_i = \uu_j \uu_i \uu_j$ when $i$ and $j$ are adjacent.  Since the braid relations are exactly those of the corresponding affine symmetric group, we may index nilCoxeter elements by elements of the affine symmetric group, e.g., $\uu_w = \uu_{i_1} \uu_{i_2}\cdots \uu_{i_k}$, whenever $s_{i_1}s_{i_2}\cdots s_{i_k}$ is a reduced word for $w$.

By work of Peterson \cite{Pet}, there is an injective ring homomorphism $j_0: \mathbf{h}_*(\Gr) \hookrightarrow \mathbb{A}$.  This map is an isomorphism on its image (actually a Hopf algebra isomorphism) $j_0: \mathbf{h}_*(\Gr) \to \mathbb{B}$, where $\mathbb{B}$ is known as the \emph{affine Fomin-Stanley subalgebra}.

\begin{defn}[Lam \cite{Lam2}]
For $w\in W^0$ corresponding to a $k$-bounded partition $\lambda$, the non-commutative $k$-Schur function $\kschur_\lambda$  is the image of the Schubert class $\xi_w$ under the isomorphism $j_0$. In other words, $\kschur_\lambda = j_0 (\xi_w)$.
\end{defn}

\subsubsection{Type A non-commutative $k$-Schur functions}
 
 We now recall the specific situation in type $A$. In \cite{Lam0}, Lam  combinatorially identified the complete homogeneous generators $\mathbf{h}_i$ inside $\mathbb{A}$.
 
 \begin{defn}
 For a subset $S \subset I$, one defines a cyclically decreasing word $w_S \in W$ to be the unique element of $W$ for which any (equivalently all) reduced words $s_{i_1} \dots s_{i_m}$ of $w_S$ satisfy:
 
 \begin{enumerate}
 \item each letter from $I$ appears at most once  in $\{i_1, \dots, i_m \}$;
 \item if $j, j+1 \in S$, then $j+1$ appears before $j$ in $i_1, \dots, i_m$ (where
     the indices are taken modulo $k+1$).
  \end{enumerate}
  Furthermore, we let $\uu_S = \uu_{w_S}$ and
  \[\displaystyle \mathbf{h}_i = \sum_{\substack{S \subset I \\ |S| = i}} \uu_S \in \mathbb{A}.\]
\end{defn}

\begin{eg}
Let $k=3$. The cyclically decreasing elements of length $2$ in the
alphabet $\{\uu_0,\uu_1,\uu_2,\uu_3\}$ are $\uu_2\uu_1$, $\uu_1\uu_0$, $\uu_0\uu_3$, $\uu_3\uu_2$, $\uu_0\uu_2$, and $\uu_1\uu_3$. Thus,
\[
 \mathbf{h}_2 = \uu_2 \uu_1 + \uu_1 \uu_0 +  \uu_0 \uu_3 + \uu_3 \uu_2 + \uu_0\uu_2 + \uu_1\uu_3.
\]
\end{eg}

\begin{thm}[Lam \cite{Lam0}]
    The elements $\{\mathbf{h}_i\}_{i \leq k}$ commute and freely generate the affine
    Fomin-Stanley    subalgebra $\mathbb{B}$ of $\mathbb{A}$. Consequently,
    \[\mathbb{B} \cong \Lambda_{(k)} := \mathbb{Q}[h_1, \dots, h_k],\]
    where $h_i$ denotes the $i^{th}$ complete homogeneous symmetric function.
\end{thm}

The type $A$ non-commutative $k$-Schur function $\kschur_\lambda$ is the image of the Schubert class $\xi_{w_\lambda}$ under $j_0$. For our purposes, we take instead the following equivalent definition (see \cite[Definition 6.5]{Lam2} and \cite[Theorem 4.6]{LSS}).

\begin{defn} The \emph{(type $A$) non-commutative $k$-Schur function} corresponding to a $k$-bounded partition $\lambda$ is the unique element $\kschur_\lambda = \sum_w c_w \uu_w$ of $\mathbb{B}$ satisfying: 
\begin{align}
& c_{w_\lambda} = 1;\label{eqn:coeff1} \\
& c_w = 0  \text{ for all other } w \in W^0.\label{eqn:coeff0}
\end{align}
\end{defn}

\section{The Lam-Shimozono up and down operators}
\label{sec:LSgraphs}
\label{sec:algebraic}

In \cite{LS}, Lam and Shimozono studied two graded graphs whose vertex set is
the affine Weyl group $W$, from which one constructs two closely-related
operators defined on the group algebra of $W$. In this section, we recall the
construction of these operators and then develop some properties of the
corresponding induced operators on the nilCoxeter algebra $\mathbb{A}$.

\subsection{Dual graded graphs}

In \cite{Fo1} and \cite{Fo2}, Fomin introduced the notion of \emph{dual graded graphs}, generalizing the notion of \emph{differential posets} in \cite{Sta}. A \emph{graded graph} is a triple $\Gamma = (V, \rho, m,E)$ where $V$ is a set of vertices, $\rho$ is a rank function on $V$, $E$ is a multiset of edges $(x,y)$ for $x,y \in V$ where $\rho(y) = \rho(x)+1$, and every edge has multiplicity $m(x,y) \in \Z_{\ge 0}$. The set of vertices of the same rank is called a \emph{level}.

$\Gamma$ is \emph{locally finite} if every $v \in V$ has finite degree, and we assume this condition for all graphs in this paper. For a graded graph $\Gamma$, the linear down and up operators $D,U: \Z V \rightarrow \Z V$ are defined as follows.
\[
 D_\Gamma(v) = \sum_{(u,v) \in E} m(u,v) u \qquad  U_\Gamma(v) = \sum_{(v,u) \in E} m(v,u) u
\]
In other words, $D$ (respectively $U$) maps a vertex $v$ to a linear combination of its neighbors in the level immediately below (respectively above) $v$ where the coefficients are the multiplicities of the edges.

A pair of graded graphs $(\Gamma,\Gamma')$ is called \emph{dual} if they have the same set of vertices and same rank function, but possibly different edges and multiplicities, and satisfies the following (Heisenberg) commutation relation
\[
 D_{\Gamma'} U_\Gamma - U_\Gamma D_{\Gamma'} = r \textrm{Id}
\]
for a fixed $r \in \Z_{\ge 0}$, called the \emph{differential coefficient}.

One can find many examples of dual graded graphs in \cite{Fo1} and \cite{Fo2},
such as the Young, Fibonacci, and Pascal lattices, the graphs of Ferrers shapes
and shifted shapes, and many more.

\subsection{The Lam-Shimozono dual graded graphs in affine type $A$}

In \cite{LS}, Lam and Shimozono introduced pairs of dual graded graphs for
arbitrary Kac-Moody algebras. Here, we specialize to the case of affine type $A_k^{(1)}$.

Following \cite{LS}, we define two graded graph structures on $W$ (see Figure
\ref{fig:truncatedgraphs} for an illustration of these graphs). The first
constructs a graph with an edge from $v$ to $w$ whenever we have a weak cover
$v \weakcover w$. We denote this graph by $\Gamma_w$ (because its edges depend on
weak Bruhat order).
The second construction uses strong order. We fix a dominant integral weight
$\Lambda\in P^+$ and let $\Gamma_s(\Lambda)$ be the graph that has $\langle
\alpha_{v,w}^\vee, \Lambda \rangle$ edges between $v$ and $w$ whenever $v
\lessdot w$.

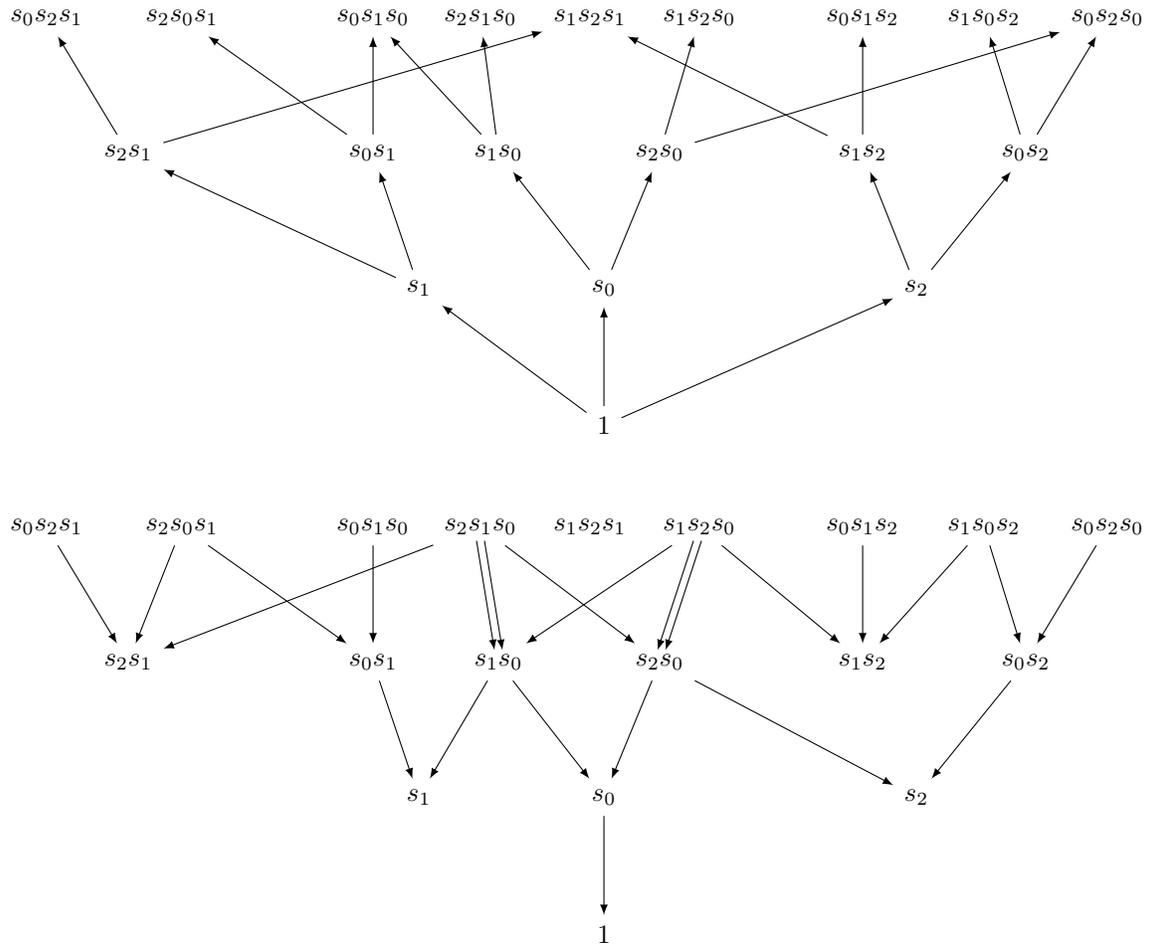
\begin{figure}[htbp]
\begin{center}
\begin{tikzpicture}[>=latex,line join=bevel,scale=0.75,every node/.style={font=\footnotesize}]
  \node (u0*u2*u1) at (18bp,212bp) [draw,draw=none] {$s_{0}s_{2}s_{1}$};
  \node (u2*u0*u1) at (86bp,212bp) [draw,draw=none] {$s_{2}s_{0}s_{1}$};
  \node (u0*u1*u0) at (182bp,212bp) [draw,draw=none] {$s_{0}s_{1}s_{0}$};
  \node (u2*u1*u0) at (236bp,212bp) [draw,draw=none] {$s_{2}s_{1}s_{0}$};
  \node (u1*u2*u1) at (291bp,212bp) [draw,draw=none] {$s_{1}s_{2}s_{1}$};
  \node (u1*u2*u0) at (346bp,212bp) [draw,draw=none] {$s_{1}s_{2}s_{0}$};
  \node (u0*u1*u2) at (428bp,212bp) [draw,draw=none] {$s_{0}s_{1}s_{2}$};
  \node (u1*u0*u2) at (489bp,212bp) [draw,draw=none] {$s_{1}s_{0}s_{2}$};
  \node (u0*u2*u0) at (551bp,212bp) [draw,draw=none] {$s_{0}s_{2}s_{0}$};
  \node (u0*u2) at (510bp,144bp) [draw,draw=none] {$s_{0}s_{2}$};
  \node (u1*u2) at (428bp,144bp) [draw,draw=none] {$s_{1}s_{2}$};
  \node (u2*u0) at (326bp,144bp) [draw,draw=none] {$s_{2}s_{0}$};
  \node (u1*u0) at (245bp,144bp) [draw,draw=none] {$s_{1}s_{0}$};
  \node (u0*u1) at (182bp,144bp) [draw,draw=none] {$s_{0}s_{1}$};
  \node (u2*u1) at (59bp,144bp) [draw,draw=none] {$s_{2}s_{1}$};
  \node (u1) at (205bp,76bp) [draw,draw=none] {$s_{1}$};
  \node (u0) at (298bp,76bp) [draw,draw=none] {$s_{0}$};
  \node (u2) at (455bp,76bp) [draw,draw=none] {$s_{2}$};
  \node (1) at (298bp,7bp) [draw,draw=none] {$1$};
  \definecolor{strokecol}{rgb}{0.0,0.0,0.0};
  \pgfsetstrokecolor{strokecol}
  \draw [black,->] (u2*u0) -- (u1*u2*u0);
  \draw [black,->] (u1*u0) -- (u2*u1*u0);
  \draw [black,->] (u0) -- (u1*u0);
  \draw [black,->] (u0) -- (u2*u0);
  \draw [black,->] (1) -- (u0);
  \draw [black,->] (u1*u0) -- (u0*u1*u0);
  \draw [black,->] (u0*u1) -- (u0*u1*u0);
  \draw [black,->] (u1) -- (u0*u1);
  \draw [black,->] (u0*u1) -- (u2*u0*u1);
  \draw [black,->] (u2*u0) -- (u0*u2*u0);
  \draw [black,->] (u0*u2) -- (u0*u2*u0);
  \draw [black,->] (u2) -- (u0*u2);
  \draw [black,->] (u0*u2) -- (u1*u0*u2);
  \draw [black,->] (1) -- (u2);
  \draw [black,->] (1) -- (u1);
  \draw [black,->] (u1) -- (u2*u1);
  \draw [black,->] (u2) -- (u1*u2);
  \draw [black,->] (u2*u1) -- (u1*u2*u1);
  \draw [black,->] (u1*u2) -- (u1*u2*u1);
  \draw [black,->] (u2*u1) -- (u0*u2*u1);
  \draw [black,->] (u1*u2) -- (u0*u1*u2);
\end{tikzpicture}

\vspace{2em}

\begin{tikzpicture}[>=latex,line join=bevel,scale=0.75,every node/.style={font=\footnotesize}]
  \node (u0*u2*u1) at (18bp,212bp) [draw,draw=none] {$s_{0}s_{2}s_{1}$};
  \node (u2*u0*u1) at (86bp,212bp) [draw,draw=none] {$s_{2}s_{0}s_{1}$};
  \node (u0*u1*u0) at (182bp,212bp) [draw,draw=none] {$s_{0}s_{1}s_{0}$};
  \node (u2*u1*u0) at (236bp,212bp) [draw,draw=none] {$s_{2}s_{1}s_{0}$};
  \node (u1*u2*u1) at (291bp,212bp) [draw,draw=none] {$s_{1}s_{2}s_{1}$};
  \node (u1*u2*u0) at (346bp,212bp) [draw,draw=none] {$s_{1}s_{2}s_{0}$};
  \node (u0*u1*u2) at (428bp,212bp) [draw,draw=none] {$s_{0}s_{1}s_{2}$};
  \node (u1*u0*u2) at (489bp,212bp) [draw,draw=none] {$s_{1}s_{0}s_{2}$};
  \node (u0*u2*u0) at (551bp,212bp) [draw,draw=none] {$s_{0}s_{2}s_{0}$};
  \node (u0*u2) at (510bp,144bp) [draw,draw=none] {$s_{0}s_{2}$};
  \node (u1*u2) at (428bp,144bp) [draw,draw=none] {$s_{1}s_{2}$};
  \node (u2*u0) at (326bp,144bp) [draw,draw=none] {$s_{2}s_{0}$};
  \node (u1*u0) at (245bp,144bp) [draw,draw=none] {$s_{1}s_{0}$};
  \node (u0*u1) at (182bp,144bp) [draw,draw=none] {$s_{0}s_{1}$};
  \node (u2*u1) at (59bp,144bp) [draw,draw=none] {$s_{2}s_{1}$};
  \node (u1) at (205bp,76bp) [draw,draw=none] {$s_{1}$};
  \node (u0) at (298bp,76bp) [draw,draw=none] {$s_{0}$};
  \node (u2) at (455bp,76bp) [draw,draw=none] {$s_{2}$};
  \node (1) at (298bp,7bp) [draw,draw=none] {$1$};
  \draw [black,->] (u2*u0) ..controls (360.19bp,125.98bp) and (411.18bp,99.098bp)  .. (u2);
  \definecolor{strokecol}{rgb}{0.0,0.0,0.0};
  \pgfsetstrokecolor{strokecol}
  \draw [black,->] (u2*u0*u1) -- (u2*u1);
  \draw [black,->] (u0*u2*u0) -- (u0*u2);
  \draw [black,->] (u2*u0) -- (u0);
  \draw [black,->] (u1*u0*u2) -- (u0*u2);
  \draw [black,->] (u1*u2*u0) -- (u1*u2);
  \draw [black,->] (u1*u0) -- (u0);
  \draw [black,->] (u0*u1*u0) -- (u0*u1);
  \draw [black,->] (u0) -- (1);
  \draw [black,->] (u1*u0) -- (u1);
  \draw [black,->] (u1*u0*u2) -- (u1*u2);
  \draw [black,->] (u0*u1) -- (u1);
  \draw [black,->] (u1*u2*u0) -- (u1*u0);
  \draw [black,->] (234bp,205bp) -- (243bp,150bp);
  \draw [black,->] (238bp,205bp) -- (247bp,150bp);
  \draw [black,->] (347bp,205bp) -- (329bp,150bp);
  \draw [black,->] (343bp,205bp) -- (325bp,150bp);
  \draw [black,->] (u2*u0*u1) -- (u0*u1);
  \draw [black,->] (u2*u1*u0) -- (u2*u0);
  \draw [black,->] (u0*u1*u2) -- (u1*u2);
  \draw [black,->] (u2*u1*u0) -- (u2*u1);
  \draw [black,->] (u0*u2*u1) -- (u2*u1);
  \draw [black,->] (u0*u2) -- (u2);
\end{tikzpicture}
\end{center}
\caption{The pair of dual graphs $\Gamma_w$ and $\Gamma_s(\Lambda_0)$,
truncated at the elements of length $3$, corresponding to the weak and strong
orders, respectively, for $k=2$.}
\label{fig:truncatedgraphs}
\end{figure}

The up and down operators for the dual graded graphs $\Gamma_w$ and
$\Gamma_s(\Lambda)$ induce operators on $\mathbb{A}$.
Specifically, define $U$ using the up operator on $\Gamma_w$,
\[ U(\uu_w) = \sum_{v \prec\!\cdot w} \uu_v,\]
and define $D_\Lambda$ using the down operator on $\Gamma_s(\Lambda)$,
\[D_\Lambda(\uu_w) = \sum_{v \lessdot w} \langle \alpha_{v,w}^\vee, \Lambda \rangle \, \uu_v. \]

It is clear from the definition and the bilinearity of the pairing $\langle \cdot ,\cdot \rangle$ that $D_{\Lambda_i + \Lambda_j} = D_{\Lambda_i} + D_{\Lambda_j}$. With this in mind, we will assume throughout this paper that $\Lambda$ is a fundamental weight. 

\begin{remark}
    Note that the operator $U$ can be realized as left-multiplication by
    $\mathbf{h}_1$ on $\mathbb{A}$. With this in mind, we define more
    generally $U_i(\uu) := \mathbf{h}_i \uu $ for $\uu \in \mathbb{A}$.
\end{remark}


\begin{remark}
Our notation differs slightly from that of \cite{LS}. Lam and Shimozono defined the operators $U$ and $D$ as operators on the opposite graphs; $D$ was defined on the weak order graph, and $U$ was defined on the strong order graph. 
Also, they define the weak order graph as depending on a central element from $\h$.
Here we just assume that this element is $c$, as it is the unique central element up to a scalar. 

\end{remark}

\begin{thm}[\cite{LS}, Theorem 2.3]
    \label{thm:dgg}
For a fundamental weight $\Lambda$, the graphs $\Gamma_w$ and $\Gamma_s(\Lambda)$ are dual graded graphs with differential coefficient 1. In other words, $D_\Lambda U - UD_\Lambda = Id$. 
\end{thm}

\begin{eg}\label{ex:dualgraphs}
    Fix $k=2$ and let $\Lambda = \Lambda_0$. Using Figure
    \ref{fig:truncatedgraphs}, we compute
    \begin{eqnarray*}
     D_\Lambda(U(\uu_1\uu_0)) \ =& D_\Lambda(\uu_0 \uu_1 \uu_0)  + D_\Lambda(\uu_2 \uu_1 \uu_0) &= (\uu_0 \uu_1) + (2 \cdot \uu_1 \uu_0 + \uu_2 \uu_0 + \uu_2 \uu_1),
     \\
     U(D_\Lambda(\uu_1\uu_0)) \ =& U(\uu_0) + U(\uu_1) &= (\uu_1 \uu_0 + \uu_2 \uu_0) + (\uu_0 \uu_1 + \uu_2 \uu_1),
    \end{eqnarray*}
    whence we conclude that $(D_\Lambda U-UD_\Lambda) (\uu_1\uu_0) = \uu_1\uu_0$.
 
\end{eg}

\subsection{Properties of the Lam-Shimozono down operator}

We start this section with several lemmas important to the main theorems of this section. The first lemma is a type $A$ version of Lam, Shilling and Shimozono \cite[Proposition 6.1]{LSS} and Pon \cite[Proposition 5.17]{P}.

\begin{lemma}\label{lemma:coeff}
For every $T \subset \{0,1, \dots, k \}$,
 \[ \displaystyle \sum_{\substack{T\subset S \\ |S| = |T| + 1}} \alpha^\vee_{w_T, w_S} = c.\]
 \end{lemma}

\begin{proof}
Let $j = S \setminus T$ and let $i$ be such that $[i,j] \subset S$ but $i-1 \not \in S$. One may check that 
$ w_T^{-1}w_S = 
                s_i s_{i+1}\cdots s_j \cdots s_{i+1} s_i$.
                Therefore $\alpha_{w_T, w_S}^\vee = \alpha_i^\vee + \alpha_{i+1}^\vee + \dots + \alpha_{j-1}^\vee + \alpha_j^\vee$. Hence \[ \sum_{\substack{T\subset S \\ |S| = |T| + 1}} \alpha^\vee_{w_T, w_S} = \sum_{j \not \in T} \alpha_{w_T, w_{T\cup{j}}}^\vee = \sum_{j \not \in T} \alpha_i^\vee + \dots + \alpha_j^\vee =\sum_i \alpha_i^\vee = c.\]

\end{proof}

\begin{lemma}\label{lemma:coroots}
Let $u,v,w \in W$ with $v \lessdot w$. Then
\begin{align*}
\alpha_{uv,uw}^\vee = \alpha_{v,w}^\vee;\\
\alpha_{vu, wu}^\vee = u^{-1}\alpha_{v,w}^\vee.
\end{align*}
\end{lemma}

\begin{proof}
For the first statement, \[s_{\alpha_{uv, uw}} = (uv)^{-1} uw = v^{-1} w = s_{\alpha_{v,w}}.\]
For the second statement \[s_{\alpha_{vu, wu}} = (vu)^{-1}wu = u^{-1} v^{-1} w u = u^{-1} s_{\alpha_{v,w}} u =  s_{u^{-1}(\alpha_{v,w})}.\] 
\end{proof}

\begin{lemma}\label{lemma:fixedaction}
If $w \in W_j$ then $w \Lambda_j = \Lambda_j$. Furthermore, if $v,w \in W_j$, then $\langle\alpha_{v,w}^\vee, \Lambda_j \rangle= 0$.
\end{lemma}

\begin{proof}
For the first statement it suffices to show this on generators. But $s_i(\Lambda_j) = \Lambda_j$ if $i \neq j$ by Equation (\ref{eqn:action}). For the second statement, if $s_{\alpha_{v,w}} = v^{-1}w = u s_i u^{-1}$ for some $u \in W$ then $\alpha_{v,w}^\vee = u (\alpha_i^\vee)$ so \[\langle \alpha_{v,w}^\vee, \Lambda_j \rangle = \langle u (\alpha_i^\vee) , \Lambda_j \rangle = \langle  \alpha_i^\vee , u^{-1} \Lambda_j \rangle = \langle \alpha_i^\vee, \Lambda_j \rangle = 0\] by the orthogonality of simple coroots and fundamental weights.
\end{proof}

We further develop properties of the operator $D_\Lambda$. Our
first observation is a generalization of the Heisenberg relation in Theorem
\ref{thm:dgg}.

\begin{thm}
    \label{thm:bracket}
    Let $\Lambda$ be a fundamental weight.
    For all $w \in W$,
    $$D_\Lambda(\mathbf{h}_i\uu_w) = \mathbf{h}_{i-1} \uu_w + \mathbf{h}_i
    D_\Lambda(\uu_w).$$
    In particular, $D_\Lambda(\mathbf{h}_i) = \mathbf{h}_{i-1}$
    and
        $$D_\Lambda \circ U_i - U_i \circ D_\Lambda = U_{i-1}.$$
\end{thm}

\begin{proof} We follow the technique of Lam and Shimozono \cite[Theorem 2.3]{LS}.
We compute \[\displaystyle D_\Lambda( \mathbf{h}_i \uu_w ) = \sum_{\substack{S \subset I \\ |S| = i}}D_\Lambda( \uu_S \uu_w) = \sum_{\substack{S\subset I\\ |S| = i}} \sum_{v \lessdot w_Sw} \langle \alpha_{v, w_Sw}^\vee, \Lambda \rangle \uu_v = \] \[ \sum_{\substack{S\subset I\\ |S| = i}} \sum_{v' \lessdot w} \langle \alpha_{w_Sv', w_Sw}^\vee, \Lambda \rangle \uu_{w_Sv'} + \sum_{\substack{S\subset I\\ |S| = i}} \sum_{\substack{T \subset S \\ |T| = |S|-1 }} 
\langle \alpha_{w_Tw, w_Sw}^\vee, \Lambda \rangle \uu_{w_Tw}\]

We deal with the two summands individually. 
 
By the first statement of Lemma \ref{lemma:coroots}, the first summand becomes \[  \sum_{\substack{S\subset I\\ |S| = i}} \sum_{v' \lessdot w} \langle \alpha_{v', w}^\vee, \Lambda \rangle \uu_{w_Sv'} =  \sum_{\substack{S\subset I\\ |S| = i}} w_S  \sum_{v' \lessdot w} \langle \alpha_{v', w}^\vee, \Lambda \rangle \uu_{v'} = \mathbf{h}_i D_\Lambda(\uu_w).\] 

By the second statement of Lemma \ref{lemma:coroots}, the second summand becomes
\[ 
\sum_{\substack{S\subset I\\ |S| = i}} \sum_{\substack{T \subset S \\ |T| = |S|-1 }} 
\langle w^{-1}\alpha_{w_T, w_S}^\vee, \Lambda \rangle \uu_{w_Tw} = \sum_{\substack{S\subset I\\ |S| = i}} \sum_{\substack{T \subset S \\ |T| = |S|-1 }} 
\langle \alpha_{w_T, w_S}^\vee, w\Lambda \rangle \uu_{w_Tw}
 = \]
 \[ 
 \sum_{\substack{ T \subset I \\|T| = i-1}} \sum_{\substack{T \subset S \\ |S| = |T|+1}} \langle \alpha_{w_T, w_S}^\vee, w\Lambda \rangle  \uu_T \uu_w = \sum_{\substack{ T \subset I \\|T| = i-1}} \langle \sum_{\substack{T \subset S\\ |S| = |T| + 1}} \alpha_{w_T, w_S}^\vee, w\Lambda \rangle  \uu_T \uu_w= 
 \]
 \[ 
 \sum_{\substack{ T \subset I \\ |T| = i-1}} \langle c, w \Lambda \rangle  \uu_T \uu_w =  \sum_{\substack{T \subset I \\ |T| = i-1}} \uu_T  \uu_w=  \mathbf{h}_{i-1} \uu_w.
 \]
Here we used Lemma \ref{lemma:coeff} and Equation (\ref{eqn:w}).
\end{proof}

Next, we study the restrictions of the operators $D_\Lambda$ to the affine
Fomin-Stanley subalgebra $\mathbb{B}$. The following theorem implies that
although the operators $D_\Lambda$, for distinct fundamental weights $\Lambda$,
are distinct on $\mathbb{A}$, their restrictions to the affine Fomin-Stanley
subalgebra $\mathbb{B}$ coincide. In fact, the action of $D_\Lambda$ on
$\mathbb{B}$ is determined by the conditions that $D_\Lambda$ is a derivation
and $D_\Lambda(\mathbf{h}_i) = \mathbf{h}_{i-1}$.

\begin{thm}
    \label{thm:derivation}
    Let $\Lambda$ be a fundamental weight.
    $D_\Lambda$ is a derivation on the affine Fomin-Stanley subalgebra
    $\mathbb{B}$. Explicitly, for $x,y \in \mathbb{B}$,
    $$D_\Lambda(\uu_{xy}) = D_\Lambda(\uu_x) \uu_y + \uu_x D_\Lambda(\uu_y).$$
    In particular, $D_\Lambda$ stabilizes
    $\mathbb{B}$; that is, $D_\Lambda(\mathbb{B}) \subset \mathbb{B}.$
\end{thm}

\begin{proof}
It is enough to prove this for a basis of $\mathbb{B}$; we will show it for the basis $\{ \mathbf{h}_\lambda := \mathbf{h}_{\lambda_1} \cdots \mathbf{h}_{\lambda_m} : \lambda_i \leq k\}$. For this, it is enough to show that $D_\Lambda(\mathbf{h}_\lambda) = \sum_{j=1}^m \mathbf{h}_{\lambda_1} \cdots D_\Lambda(\mathbf{h}_{\lambda_j}) \cdots \mathbf{h}_{\lambda_m}$. However this follows by the linearity of $D_\Lambda$  and  Lemma \ref{thm:bracket}:
\[ D_\Lambda(\mathbf{h}_{\lambda_1}\cdots \mathbf{h}_{\lambda_m}) = \mathbf{h}_{\lambda_1} D_\Lambda(\mathbf{h}_{\lambda_2}\cdots  \mathbf{h}_{\lambda_m}) +  \mathbf{h}_{\lambda_1-1} \mathbf{h}_{\lambda_2}\cdots  \mathbf{h}_{\lambda_m} = \]
\[ \mathbf{h}_{\lambda_1}  \mathbf{h}_{\lambda_2} D_\Lambda( \mathbf{h}_{\lambda_3}\cdots \mathbf{h}_{\lambda_m}) + \mathbf{h}_{\lambda_1} \mathbf{h}_{\lambda_2-1} \cdots \mathbf{h}_{\lambda_m} +  \mathbf{h}_{\lambda_1-1} \mathbf{h}_{\lambda_2}\cdots  \mathbf{h}_{\lambda_m}\] 
\[= \dots = \sum_{j=1}^m \mathbf{h}_{\lambda_1} \cdots D(\mathbf{h}_{\lambda_j}) \cdots \mathbf{h}_{\lambda_m}\]
\end{proof}

Finally, we describe the coefficients of the operator combinatorially. The next
result shows that it suffices to know the value of $D_\Lambda$ on the elements
in $W^j$. In the case that $j=0$, this says that it suffices to know the values
of $D_\Lambda$ on the affine Grassmannian elements.
\begin{thm}
    \label{thm:scalars}
    Suppose $w \in W^j$ and $v \in W_j$. Then 
    \[ D_{\Lambda_j}(\uu_{wv}) = D_{\Lambda_j}(\uu_{w}) \uu_{v}.\]
\end{thm}

\begin{proof}
We compute 
\begin{align*}
 D_{\Lambda_j}(\uu_{xy}) & = \sum_{v \lessdot x } \langle \alpha_{vy, xy} \rangle \uu_{vy}+ \sum_{w \lessdot y} \langle \alpha_{xw, xy}, \Lambda_j \rangle \uu_{xw}  \\ 
& = \sum_{v \lessdot x } \langle y^{-1}(\alpha_{v, x}), \Lambda_j \rangle \uu_{vy}+ \sum_{w \lessdot y} \langle \alpha_{w, y}, \Lambda_j \rangle \uu_{xw} & \textrm{ by Lemma \ref{lemma:coroots}}\\
& = \sum_{v \lessdot x} \langle \alpha_{v,x}^\vee, y\Lambda_j \rangle \uu_{vy} + \sum_{w \lessdot y} \langle \alpha_{w,y}^\vee, \Lambda_j \rangle \uu_{xw} & \textrm{ by Equation (\ref{eqn:w}})\\
& =\sum_{v \lessdot x} \langle \alpha_{v,x}^\vee, \Lambda_j \rangle \uu_{vy}& \textrm{ by Lemma \ref{lemma:fixedaction} }\\
& = D_{\Lambda_j}(\uu_x) \uu_y & \qedhere
\end{align*} 
\end{proof}

We now give a combinatorial formula to apply the down operator to the elements
of $W^j$. This generalizes the description of the coefficients
given in \cite{LLMS}.

\begin{thm}
    \label{thm:D on grassmannians}
    Suppose $w \in W^j$. Then 
    \[D_{\Lambda_j}(\uu_w) = \sum_{z \lessdot w} c_z^{w,j} \uu_z,\]
    where $c_z^{w, j}$ is the number of addable $(i_\ell-j)$-cells of the
    $(k+1)$-core $s_{i_{\ell-1} -j} \cdots s_{i_1 -j} \emptyset$,
    where $s_{i_m} \cdots s_{i_2} s_{i_1}$ is a reduced expression
    for $w$ and $s_{i_m} \cdots \widehat{s_{i_\ell}} \cdots s_{i_1}$ is a
    reduced expression for $z$.
\end{thm}

\begin{proof}
Suppose $v \lessdot w$ with $w \in W^j$. Let $x = s_{i_{\ell-1}} \cdots s_{i_1}$ and $y = s_{i_\ell} \cdots s_{i_1}$. Then $\alpha_{v,w}^\vee = \alpha_{x,y}^\vee$ since $v^{-1} w = x^{-1} y$. By \cite{LLMS} and \cite{LS}, $\langle \alpha_{x,y}^\vee, \Lambda \rangle$ is the number of ribbons associated to the cover $x \lessdot y$. However, $x \weakcover y$, so this is equivalent to the number of cells added between the corresponding shapes $ s_{i_{\ell-1} -j}\cdots s_{i_1 - j}\emptyset$ and $ s_{i_\ell - j}\cdots s_{i_1 - j}\emptyset$.
\end{proof}

These previous two theorems combine to give a combinatorial method for calculating the
down operator on any basis element $\uu_w$. We illustrate this in the following
example.

\begin{eg}\label{eg:Dop}
Fix $k= 3$.
We calculate $D_{\Lambda_0}(\uu_2 \uu_3 \uu_0 \uu_1 \uu_2 \uu_3 \uu_0 \uu_2)$.
By Theorem \ref{thm:scalars},
$$
D_{\Lambda_0}(\uu_2\uu_3\uu_0\uu_1\uu_2\uu_3\uu_0\uu_2)
=
D_{\Lambda_0}(\uu_2\uu_3\uu_0\uu_1\uu_2\uu_3\uu_0)\uu_2
$$
since $s_2 s_3 s_0 s_1 s_2 s_3 s_0 \in W^0$.
Hence, it suffices to calculate
$D_{\Lambda_0}(\uu_2\uu_3\uu_0\uu_1\uu_2\uu_3\uu_0)$.

By Theorem \ref{thm:D on grassmannians},
the coefficient of $\uu_2\uu_3\uu_1\uu_2\uu_3\uu_0$
in
$D_{\Lambda_0}(\uu_2\uu_3\uu_0\uu_1\uu_2\uu_3\uu_0)$
is the number of addable $0$-cells in the $4$-core $s_1 s_2 s_3 s_0 \cdot \emptyset =
(2,1,1,1)$, which is $2$ (as indicated by the shaded cells in Figure \ref{fig:addable}).

\begin{figure}[ht]
\begin{center}
    $\tikztableausmall{{0,1},{3,\boldentry X},{2},{1},{\boldentry X},}$
\end{center}
\caption{The addable $0$-cells in the $4$-core $(2,1,1,1)$.}
\label{fig:addable}
\end{figure}
Similarly, one computes all the other coefficients:
\begin{align*}
D_{\Lambda_0}(\uu_2\uu_3\uu_0\uu_1\uu_2\uu_3\uu_0) 
&= 3\,\uu_3\uu_0\uu_1\uu_2\uu_3\uu_0 + 2\,\uu_2\uu_0\uu_1\uu_2\uu_3\uu_0 + 2\,\uu_2\uu_3\uu_1\uu_2\uu_3\uu_0 \\
&\phantom{=} + \uu_2\uu_3\uu_0\uu_1\uu_3\uu_0 + \uu_2\uu_3\uu_0\uu_1\uu_2\uu_0 + \uu_2\uu_3\uu_0\uu_1\uu_2\uu_3.
\end{align*}
\end{eg}

\section{Expansions of non-commutative $k$-Schur functions and $k$-Littlewood--Richardson
coefficients for ``near'' rectangles}

This section describes the connection between expansions of non-commutative $k$-Schur functions
in the standard basis of $\mathbb{A}$ and the $k$-Littlewood--Richardson rule.
We then recall the expansions of the non-commutative $k$-Schur functions for $k$-rectangles,
from which we deduce expansions of the non-commutative $k$-Schur functions for the ``near''
rectangles.

\subsection{Expansion of $\mathfrak{s}_\lambda^{(k)}$ and the $k$-Littlewood--Richardson coefficients}
\label{ss:expansions}

Under the identification of $\mathbb{B}$ and $\Lambda_{(k)}$, we let $\kschur_\lambda$ correspond to $s_\lambda^{(k)}$.
An important problem in the theory of $k$-Schur functions is to understand the multiplicative structure coefficients $c_{\lambda, \mu}^{\nu, (k)}$, called the \textit{$k$-Littlewood--Richardson coefficients}:

\[ s_\lambda^{(k)} s_\mu^{(k)} = \sum_\nu c_{\lambda, \mu}^{\nu, (k)} s_\nu^{(k)}.\]

Another difficult problem is determining an expansion  for $\kschur_\lambda$ in terms of the natural basis $\{\uu_w\}_{w \in W}$ of $\mathbb{A}$. In other words, to find the coefficients $d_\lambda^w$ in the expansion:
\[ \kschur_\lambda = \sum_{w\in W} d_\lambda^w \uu_w.\]

In \cite{Lam0}, Lam proved that these two problems are actually equivalent.
Explicitly, he observed the following.
\begin{thm}[{\cite[Proposition 42]{Lam0}}]\label{obs:equiv}
The coefficient $c_{\lambda,\mu}^{\nu, (k)}$ is nonzero only if 
$w_\mu$ is less than $w_\nu$ in left weak order,
and in this case $c_{\lambda,\mu}^{\nu, (k)} = d^{w_\nu w_\mu^{-1}}_\lambda$.
\end{thm}

The main application in this paper of the down operator is to give the
coefficients $d_\lambda^w$ via explicit combinatorics when $\lambda$ is a
``near'' rectangle. From this viewpoint our result gives a combinatorial
description of the corresponding $k$-Littlewood--Richardson coefficients. A
previous result of \cite{BBTZ} will be reviewed in the next section. It
contains the combinatorics of the coefficients that appear in the expansion of
a non-commutative $k$-Schur function corresponding to a rectangle and is needed to prove our
main result.

\subsection{Expansions of rectangular non-commutative $k$-Schur functions}
In \cite{BBTZ}, Berg, Bergeron, Thomas and Zabrocki gave a combinatorial formula for the expansion of the non-commutative $k$-Schur function $\kschur_R$ indexed by a $k$-rectangle $R$. We recall their result here; it will be a stepping stone for our main result.

Let $\nu$ and $\mu$ be $k$-bounded partitions. For the skew shape $\nu/\mu$,
let $\READ{\nu/\mu} \in W$ be the word formed by the residues of the cells in
$\nu/\mu$, reading each row from right to left and taking the rows from bottom
to top. See Example \ref{eg:33}.

\begin{thm}[Berg, Bergeron, Thomas, Zabrocki \cite{BBTZ}]
Suppose $R = (c^r)$ with $c+r = k+1$.
The non-commutative $k$-Schur function $\kschur_R$ has the expansion: \[\kschur_R = \sum_{\lambda \subset R} \uu_{\READ{(R\cup\lambda)/\lambda}},\]
where $\uu_{\READ{(R\cup\lambda)/\lambda}}$ is the monomial in the generators
$\uu_i$ corresponding to ${\READ{(R\cup\lambda)/\lambda}}$.
\end{thm}


\begin{eg}\label{eg:33}
Let $R = (3,3)$ and $k=4$. Then $\mathfrak{s}_R^{(4)}$ is the sum of all the monomials in
$\uu_i$ corresponding to the reading words of the skew-partitions $(R\cup
\lambda)/\lambda$, where $\lambda$ is a partition contained inside the
rectangle $R$, as shown:

\begin{gather*}
\begin{array}{ccccc}
  \tikztableausmall{{0,1,2},{4,0,1},}
& \tikztableausmall{{X,1,2},{4,0,1},{\boldentry 3},} 
& \tikztableausmall{{X,X,2},{4,0,1},{\boldentry 3,\boldentry 4},}
& \tikztableausmall{{X,1,2},{X,0,1},{\boldentry 3},{\boldentry 2},}
& \tikztableausmall{{X,X,2},{X,0,1},{\boldentry 3,\boldentry 4},{\boldentry 2},}
\\
  \uu_1\uu_0\uu_4\uu_2\uu_1\uu_0
& \uu_3\uu_1\uu_0\uu_4\uu_2\uu_1
& \uu_4\uu_3\uu_1\uu_0\uu_4\uu_2
& \uu_2\uu_3\uu_1\uu_0\uu_2\uu_1
& \uu_2\uu_4\uu_3\uu_1\uu_0\uu_2
\\[1em]
  \tikztableausmall{{X,X,X},{4,0,1},{\boldentry 3,\boldentry 4,\boldentry 0},}
& \tikztableausmall{{X,X,2},{X,X,1},{\boldentry 3,\boldentry 4},{\boldentry 2,\boldentry 3},}
& \tikztableausmall{{X,X,X},{X,0,1},{\boldentry 3,\boldentry 4,\boldentry 0},{\boldentry 2},}
& \tikztableausmall{{X,X,X},{X,X,1},{\boldentry 3,\boldentry 4,\boldentry 0},{\boldentry 2,\boldentry 3},}
& \tikztableausmall{{X,X,X},{X,X,X},{\boldentry 3,\boldentry 4,\boldentry 0},{\boldentry 2,\boldentry 3,\boldentry 4},}
\\
  \uu_0\uu_4\uu_3\uu_1\uu_0\uu_4
& \uu_3\uu_2\uu_4\uu_3\uu_1\uu_2
& \uu_2\uu_0\uu_4\uu_3\uu_1\uu_0
& \uu_3\uu_2\uu_0\uu_4\uu_3\uu_1
& \uu_4\uu_3\uu_2\uu_0\uu_4\uu_3
\end{array}
\end{gather*}
\end{eg}

\subsection{Non-commutative $k$-Schur functions for ``near'' rectangles}

\begin{prop}\label{prop:rectangle}
Suppose $R = (c^r)$ with $c+r = k+1$ and let $S = (c^{r-1},c-1)$ be the partition obtained from $R$ by removing its bottom-right corner. Let $\Lambda$ be a fundamental weight. Then $D_{\Lambda}(\kschur_R) = \kschur_{S}$.
\end{prop}

\begin{proof}
By Theorem \ref{thm:derivation} and Equations (\ref{eqn:coeff1}), (\ref{eqn:coeff0}), it is enough to count the  terms appearing in $D_\Lambda(\kschur_R)$ which come from $W^0$. Since $D_\Lambda$ is independent of choice of $\Lambda$ on $\mathbb{B}$, we may choose $\Lambda = \Lambda_0$ for our calculations. The only term from $W^0$ which appears in the expansion of $\kschur_R$ is the element $w_R \in W^0$ corresponding to $R$. By Theorem \ref{thm:D on grassmannians} and Theorem \ref{thm:scalars}, the only terms from $W^0$ which may appear in $D_{\Lambda_0}(\kschur_R)$ are those which are strongly covered by $w_R$. However, $v$ being strongly covered by $w_R$ is equivalent to $v \emptyset \subset R$ (see for instance \cite{Lascoux, MM}). There is only one partition contained in $R$ which has size $|R|-1$; this partition is $S$.
\end{proof}

For $\lambda \subset R$ and a cell $x \in \lambda$, we let $\READ{R,\lambda,x}$
denote the word corresponding to the diagram $(R \cup \lambda_x )/ \lambda$,
where $\lambda_x$ denotes the diagram with the cell $x$ removed.

\begin{eg} Let $k = 4$, let $R = (3,3)$, $\lambda = (2,1) \subset R$ and $x = (1,2) \in \lambda$. 
Then $\READ{R,\lambda,x} = s_2 s_3 s_1 s_0 s_2$. 
\[ \tikztableausmall{{X,X,2},{X,0,1},{\boldentry 3,\boldentry X},{\boldentry 2},}\]

\end{eg}

\begin{thm}\label{thm:mainthm}
\[\kschur_{(c^{r-1}, c-1)} = \sum_{\lambda \subset R} \sum_{x \in \lambda} \uu_{\READ{R,\lambda,x}}.\]
\end{thm}

\begin{proof}
    This follows from Theorem \ref{thm:D on grassmannians}
    and the application of Proposition \ref{prop:rectangle} with the
    fundamental weight $\Lambda_c$.
\end{proof}

\begin{eg}\label{eg:32}
Let $k = 4$ and $\lambda = (3,2)$. Using Example \ref{eg:33}, we can realize $\mathfrak{s}_{3,2}^{(4)}$ as $D_{\Lambda_3}(\mathfrak{s}_{3,3}^{(4)})$. $D_{\Lambda_3}$ acts on the pictures by deleting a bold letter from a term in the expansion of $\mathfrak{s}_{3,3}^{(4)}$. In particular, the first diagram of $\mathfrak{s}_{3,3}^{(4)}$ has no bold letters, so it does not contribute any terms to $\mathfrak{s}_{3,2}^{(4)}$.

The second diagram gives a term:
$$
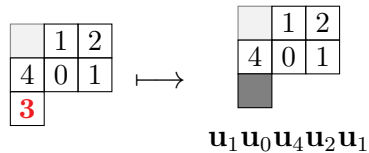
\tikztableausmall{{X,1,2},{4,0,1},{\boldentry 3},}
\longmapsto
\begin{array}{c}\tikztableausmall{{X,1,2},{4,0,1},{\boldentry X},} \\ \uu_1\uu_0\uu_4\uu_2\uu_1 \end{array}
$$
The third and fourth diagrams each give two terms:
$$
  \tikztableausmall{{X,X,2},{4,0,1},{\boldentry 3,\boldentry 4},}
  \longmapsto
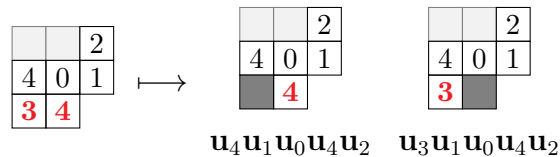
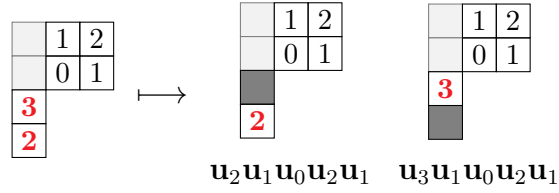
\begin{array}{cc}
  \tikztableausmall{{X,X,2},{4,0,1},{\boldentry X,\boldentry 4},}
& \tikztableausmall{{X,X,2},{4,0,1},{\boldentry 3,\boldentry X},}
\\
  \uu_4\uu_1\uu_0\uu_4\uu_2
& \uu_3\uu_1\uu_0\uu_4\uu_2
\end{array}
$$
$$
  \tikztableausmall{{X,1,2},{X,0,1},{\boldentry 3},{\boldentry 2},}
  \longmapsto
\begin{array}{cc}
  \tikztableausmall{{X,1,2},{X,0,1},{\boldentry X},{\boldentry 2},}
& \tikztableausmall{{X,1,2},{X,0,1},{\boldentry 3},{\boldentry X},}
\\
  \uu_2\uu_1\uu_0\uu_2\uu_1
& \uu_3\uu_1\uu_0\uu_2\uu_1
\end{array}
$$
The fifth and sixth diagrams gives 3 terms each:
$$
  \tikztableausmall{{X,X,2},{X,0,1},{\boldentry 3,\boldentry 4},{\boldentry 2},}
  \longmapsto
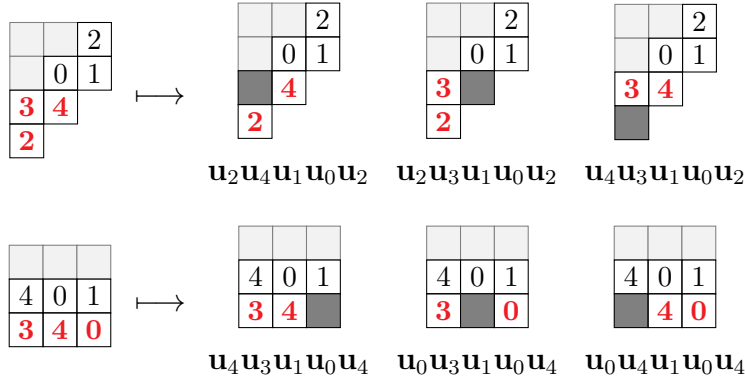
\begin{array}{ccc}
  \tikztableausmall{{X,X,2},{X,0,1},{\boldentry X,\boldentry 4},{\boldentry 2},}
& \tikztableausmall{{X,X,2},{X,0,1},{\boldentry 3,\boldentry X},{\boldentry 2},}
& \tikztableausmall{{X,X,2},{X,0,1},{\boldentry 3,\boldentry 4},{\boldentry X},}
\\
  \uu_2\uu_4\uu_1\uu_0\uu_2
& \uu_2\uu_3\uu_1\uu_0\uu_2
& \uu_4\uu_3\uu_1\uu_0\uu_2
\end{array}
$$
$$
  \tikztableausmall{{X,X,X},{4,0,1},{\boldentry 3,\boldentry 4,\boldentry 0},}
  \longmapsto
\begin{array}{ccc}
  \tikztableausmall{{X,X,X},{4,0,1},{\boldentry 3,\boldentry 4,\boldentry X},}
& \tikztableausmall{{X,X,X},{4,0,1},{\boldentry 3,\boldentry X,\boldentry 0},}
& \tikztableausmall{{X,X,X},{4,0,1},{\boldentry X,\boldentry 4,\boldentry 0},}
\\
  \uu_4\uu_3\uu_1\uu_0\uu_4
& \uu_0\uu_3\uu_1\uu_0\uu_4
& \uu_0\uu_4\uu_1\uu_0\uu_4
\end{array}
$$
The seventh and eigth diagrams give 4 terms each:
$$
  \tikztableausmall{{X,X,2},{X,X,1},{\boldentry 3,\boldentry 4},{\boldentry 2,\boldentry 3},}
  \longmapsto
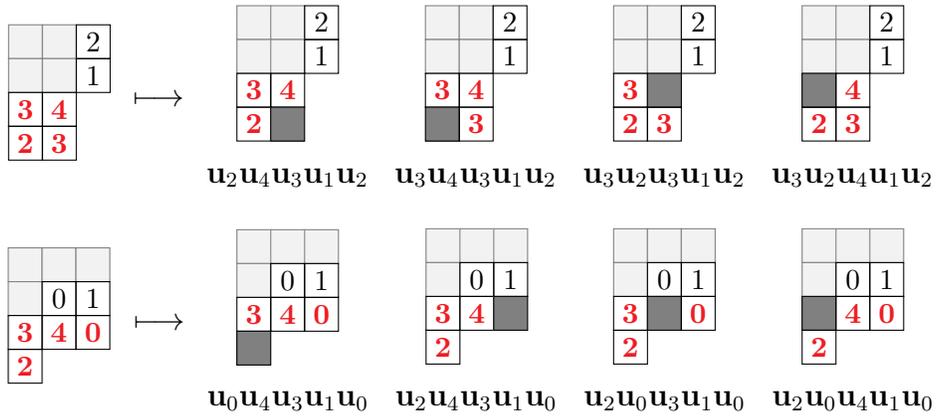
\begin{array}{cccc}
  \tikztableausmall{{X,X,2},{X,X,1},{\boldentry 3,\boldentry 4},{\boldentry 2,\boldentry X},}
& \tikztableausmall{{X,X,2},{X,X,1},{\boldentry 3,\boldentry 4},{\boldentry X,\boldentry 3},}
& \tikztableausmall{{X,X,2},{X,X,1},{\boldentry 3,\boldentry X},{\boldentry 2,\boldentry 3},}
& \tikztableausmall{{X,X,2},{X,X,1},{\boldentry X,\boldentry 4},{\boldentry 2,\boldentry 3},}
\\
  \uu_2\uu_4\uu_3\uu_1\uu_2 
& \uu_3\uu_4\uu_3\uu_1\uu_2 
& \uu_3\uu_2\uu_3\uu_1\uu_2 
& \uu_3\uu_2\uu_4\uu_1\uu_2 
\end{array}
$$
$$
  \tikztableausmall{{X,X,X},{X,0,1},{\boldentry 3,\boldentry 4,\boldentry 0},{\boldentry 2},}
  \longmapsto
\begin{array}{cccc}
  \tikztableausmall{{X,X,X},{X,0,1},{\boldentry 3,\boldentry 4,\boldentry 0},{\boldentry X},}
& \tikztableausmall{{X,X,X},{X,0,1},{\boldentry 3,\boldentry 4,\boldentry X},{\boldentry 2},}
& \tikztableausmall{{X,X,X},{X,0,1},{\boldentry 3,\boldentry X,\boldentry 0},{\boldentry 2},}
& \tikztableausmall{{X,X,X},{X,0,1},{\boldentry X,\boldentry 4,\boldentry 0},{\boldentry 2},}
\\
  \uu_0\uu_4\uu_3\uu_1\uu_0
& \uu_2\uu_4\uu_3\uu_1\uu_0
& \uu_2\uu_0\uu_3\uu_1\uu_0
& \uu_2\uu_0\uu_4\uu_1\uu_0
\end{array}
$$
The ninth diagram gives 5 terms:
$$
  \tikztableausmall{{X,X,X},{X,X,1},{\boldentry 3,\boldentry 4,\boldentry 0},{\boldentry 2,\boldentry 3},}
  \longmapsto
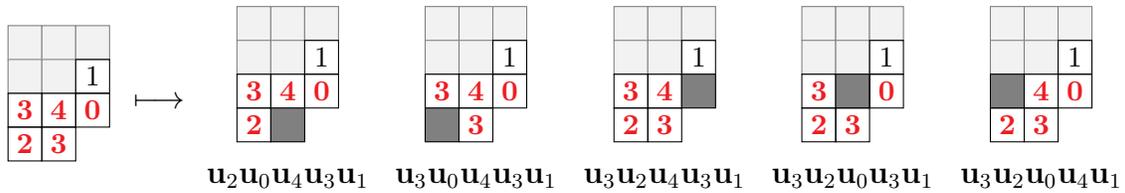
\begin{array}{ccccc}
  \tikztableausmall{{X,X,X},{X,X,1},{\boldentry 3,\boldentry 4,\boldentry 0},{\boldentry 2,\boldentry X},}
& \tikztableausmall{{X,X,X},{X,X,1},{\boldentry 3,\boldentry 4,\boldentry 0},{\boldentry X,\boldentry 3},}
& \tikztableausmall{{X,X,X},{X,X,1},{\boldentry 3,\boldentry 4,\boldentry X},{\boldentry 2,\boldentry 3},}
& \tikztableausmall{{X,X,X},{X,X,1},{\boldentry 3,\boldentry X,\boldentry 0},{\boldentry 2,\boldentry 3},}
& \tikztableausmall{{X,X,X},{X,X,1},{\boldentry X,\boldentry 4,\boldentry 0},{\boldentry 2,\boldentry 3},}
\\
  \uu_2\uu_0\uu_4\uu_3\uu_1
& \uu_3\uu_0\uu_4\uu_3\uu_1
& \uu_3\uu_2\uu_4\uu_3\uu_1
& \uu_3\uu_2\uu_0\uu_3\uu_1
& \uu_3\uu_2\uu_0\uu_4\uu_1
\end{array}
$$
The tenth and final diagram gives six terms:
$$
  \tikztableausmall{{X,X,X},{X,X,X},{\boldentry 3,\boldentry 4,\boldentry 0},{\boldentry 2,\boldentry 3,\boldentry 4},}
  \longmapsto
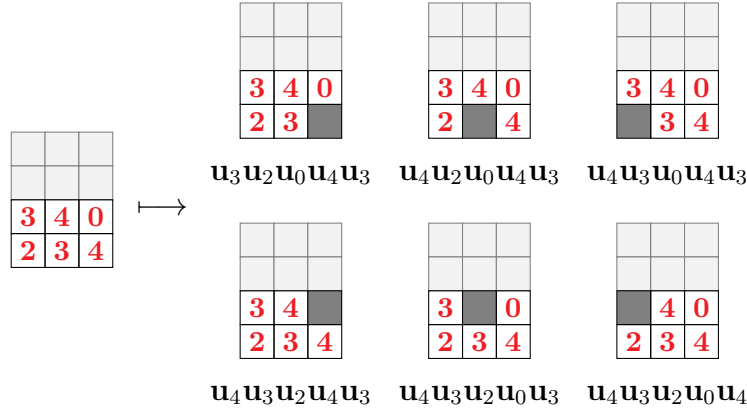
\begin{array}{ccc}
  \tikztableausmall{{X,X,X},{X,X,X},{\boldentry 3,\boldentry 4,\boldentry 0},{\boldentry 2,\boldentry 3,\boldentry X},}
& \tikztableausmall{{X,X,X},{X,X,X},{\boldentry 3,\boldentry 4,\boldentry 0},{\boldentry 2,\boldentry X,\boldentry 4},}
& \tikztableausmall{{X,X,X},{X,X,X},{\boldentry 3,\boldentry 4,\boldentry 0},{\boldentry X,\boldentry 3,\boldentry 4},}
\\
  \uu_3\uu_2\uu_0\uu_4\uu_3
& \uu_4\uu_2\uu_0\uu_4\uu_3
& \uu_4\uu_3\uu_0\uu_4\uu_3
\\[1em]
  \tikztableausmall{{X,X,X},{X,X,X},{\boldentry 3,\boldentry 4,\boldentry X},{\boldentry 2,\boldentry 3,\boldentry 4},}
& \tikztableausmall{{X,X,X},{X,X,X},{\boldentry 3,\boldentry X,\boldentry 0},{\boldentry 2,\boldentry 3,\boldentry 4},}
& \tikztableausmall{{X,X,X},{X,X,X},{\boldentry X,\boldentry 4,\boldentry 0},{\boldentry 2,\boldentry 3,\boldentry 4},}
\\
  \uu_4\uu_3\uu_2\uu_4\uu_3
& \uu_4\uu_3\uu_2\uu_0\uu_3
& \uu_4\uu_3\uu_2\uu_0\uu_4
\end{array}
$$
%
Then $\mathfrak{s}_{3,2}^{(4)}$ is a sum of the 30 words above.
\end{eg}

The following lemma is a standard result in the theory of Bruhat order on Coxeter groups.

\begin{lemma}\label{lemma:delete_letter}
If $w = s_{i_1}\cdots s_{i_m}$ is a reduced expression for $w$, $ u = s_{i_1}\cdots \widehat{s_{i_j}}\cdots s_{i_m}$ and $ v = s_{i_1}\cdots \widehat{ s_{i_\ell}}\cdots s_{i_m}$, then $u = v$ if and only if $j = \ell$.
\end{lemma}

\begin{cor}
Let $S = (c^{r-1}, c-1)$ with $c+r = k+1$. Then the coefficient $c_{\lambda, S}^{\nu, (k)}$ is either $0$ or $1$.
\end{cor}

\begin{proof}

Suppose a coefficient $c_{\lambda, S}^{\nu, (k)} > 1$. By Observation \ref{obs:equiv}, there exists two terms in the expansion of Theorem \ref{thm:mainthm} which give equivalent reduced words. Therefore, there exists distinct $x \in \lambda \subset R$ and $y \in \mu \subset R$ such that 
$\READ{R, \lambda, x} = \READ{R, \mu, y}$.

If $\lambda = \mu$ then $\READ{R, \lambda,x} = \READ{R, \lambda, y}$ implies $x = y$ by Lemma \ref{lemma:delete_letter}.

Otherwise we assume $\lambda \neq \mu$. Let $u \in W^c$ and $w_{R/\lambda} \in W_c$ be such that $uw_{R/\lambda} = \READ{(R\cup\lambda)/\lambda}$. Then  $\READ{R, \lambda,x} = u_x w_{R/\lambda}$, where $u_x$ denotes the resulting word from deleting the letter corresponding to  $x \in \lambda$ from $u$.  Similarly, let $v \in W^c$ and $w_{R/\mu} \in W_c$ be such that $v w_{R/\mu} = \READ{(R\cup\mu)/\mu}$, and $\READ{R, \mu,x} = v_y w_{R/\mu}$.

It is easy to see that $u^{-1} \in W_{2c}$. Also $w_{R/\lambda}^{-1} \in W^{2c}$ since the unique removable cell of the rectangle $R$ has residue $2c$. Similarly $v^{-1} \in W_{2c}$ and $w_{R/\mu}^{-1} \in W^{2c}$. Also $u_x^{-1}, v_y^{-1} \in W_{2c}$ so $w_{R/\lambda}^{-1} u_x^{-1}= w_{R/\mu}^{-1} v_y^{-1}$, which implies $w_{R/\lambda} = w_{R/\mu}$. Therefore $\lambda = \mu$, a contradiction.
\end{proof}

\begin{eg}

Using Example \ref{eg:32} and Theorem \ref{obs:equiv}, we compute $c_{(2,1), (3,2)}^{(3,3,1,1), (4)}$. The action of the element $\uu_2 \uu_3 \uu_1 \uu_0 \uu_2$ on the $5$-core $(2,1)$ gives the $5$-core $(4,4,1,1)$, which corresponds to the $4$-bounded partition $(3,3,1,1)$. Therefore, the coefficient $c_{(2,1), (3,2)}^{(3,3,1,1), (4)}$ is the coefficient of $\uu_2 \uu_3 \uu_1 \uu_0 \uu_2$ in the expansion of $\mathfrak{s}_{3,2}^{(4)}$, which is $1$.
\end{eg}

\section{Further directions}

As mentioned in the introduction, this paper is the introduction to a more
general family of operators $\{D_J\}$, indexed over all compositions $J$. The
operator $D_{\Lambda_0}$ studied here is one instance of these operators; it is
$D_J$ for the composition $J = [1]$. These operators are defined on the affine
nilCoxeter algebra $\mathbb{A}$ and seem to restrict nicely to the affine
Fomin-Stanley subalgebra. We will develop properties and applications of these
operators in a future article.

\section*{Acknowledgements}
We have benefited from many conversations with Nantel Bergeron, Steven Pon and Mike Zabrocki, as well as email correspondence with Thomas Lam, Jennifer Morse and Mark Shimozono. We also thank Anne Schilling for valuable comments and corrections.

This research was facilitated by computer exploration using the open-source
mathematical software \texttt{Sage}~\cite{sage} and its algebraic
combinatorics features developed by the \texttt{Sage-Combinat}
community \cite{sage-combinat}.

\bibliographystyle{abbrvnat}

\end{document}